\documentclass{amsart}

\newcommand{\IR}{\mathbb{R}}
\newcommand{\IN}{\mathbb{N}}
\newcommand{\U}{\mathcal U}
\newcommand{\e}{\varepsilon}
\newcommand{\diam}{\operatorname{diam}}
\newcommand{\supp}{\operatorname{supp}}
\newcommand{\pr}{\mathrm{pr}}
\newcommand{\limn}{\lim_{n}}
\newcommand{\HM}{\mathsf{HM}}

\newtheorem{theorem}{Theorem}
\newtheorem{proposition}{Proposition}
\newtheorem{lemma}{Lemma}
\theoremstyle{definition}
\newtheorem{question}{Question}
\newtheorem{problem}{Problem}

\title{On linear operators extending [pseudo]metrics}
\author{Taras Banakh and Czes\l aw Bessaga}
\address{Department of Mathematics, Lviv University, Universytetska 1,
Lviv, 290602, Ukraine}
\email{tbanakh@yahoo.com}
\address{Instytut Matematyki, Uniwersytet Warszawski, ul. Banacha 2,
02-097 Warszawa, Poland}
\email{bessaga@impan.gov.pl}

\begin{document}
\begin{abstract}
For every closed subset $X$ of a
stratifiable [resp. metrizable] space $Y$
we construct a positive linear extension
operator $T:\IR^{X\times X}\to \IR^{Y\times Y}$
preserving constant functions, bounded functions,
continuous functions, pseudometrics, metrics, [resp. dominating metrics,
and admissible metrics].
This operator is continuous with respect to each of the three
topologies: point-wise
convergence,  uniform, and compact-open.

An equivariant analog of the above statement is proved as well.
\end{abstract}

\maketitle

The problem of existing a linear operator extending [pseudo]metrics from a
closed subset of a metric compactum $X$ over all of $X$ was posed by the
second author in
\cite{Be1} and partly solved in \cite{Be1}, \cite{Be2}. A complete
solution of this problem appeared in \cite{Ba1} and \cite{Pi} (see also
\cite{Ba0} and \cite{Ba2}).
M.~Zarichnyi \cite{Za} presented a very simple
construction of such extension operators.

In contrast to the mentioned results, the present paper, which is
a simplified and generalized version of the preprint \cite{Ba0},
allows to construct
linear operators extending metrics which are continuous with respect
to the point-wise convergence of functions.

For a space $Z$ we denote by $\IR^Z$ the space of all, not
necessarily continuous, real-valued functions on $Z$ with the
Tychonoff product topology (which corresponds to the point-wise
convergence of the functions).

Our first theorem is quite general and concerns stratifiable spaces, see
\cite{Bo} for their definitions and properties. Here we mention only that
each metrizable space is stratifiable, each stratifiable space is
perfectly paracompact, and every subspace of a stratifiable space is
stratifiable too.

\begin{theorem}\label{T1} Suppose $Y$ is a stratifiable space and $X$ is a
closed subspace of $Y$ with $|X|\ge 2$.  There exists a positive linear
extension operator $T:\IR^{X\times X}\to \IR^{Y\times Y}$
preserving constant functions, bounded functions, continuous functions,
pseudometrics, and metrics.
This operator is continuous with respect to each of the three
topologies: point-wise convergence, uniform, and compact-open.
\end{theorem}

Obviously the phrase ``$T$ preserves bounded functions, etc.''
means that $T$ carries bounded functions, etc., on $X\times X$
into bounded functions, etc., on $Y\times Y$.

For metrizable spaces we are able to prove much more. It will be
convenient to formulate the corresponding result in terms of uniform
spaces (see Chapter 8 of \cite{En} for the theory of uniform spaces). We
remark that each metric space is automaticly a uniform space. We call a
uniform space {\it metrizable} if its uniformity is generated by a metric.

\begin{theorem}\label{T2} Suppose $Y$ is a metrizable uniform space and $X$ is a
closed subspace of $Y$ with $|X|\ge 2$.  There exists a positive linear
extension operator $T:\IR^{X\times X}\to \IR^{Y\times Y}$
preserving constant functions, bounded functions, continuous functions,
pseudometrics, metrics, admissible metrics, dominating metrics, and
uniformly dominating metrics.
This operator is continuous with respect to each of the three
topologies: point-wise convergence, uniform, and compact-open.
Moreover, if the uniform space $Y$ is complete, then $T$ preserves
complete continuous uniformly dominating metrics. If $Y$ is totally
bounded and $\dim(Y\setminus X)<\infty$, then $T$ preserves
totally bounded pseudometrics.
\end{theorem}

A metric $d$ on a topological [resp. uniform] space $Z$ is
called {\em dominating} [resp. {\em uniformly dominating\/}] if
the formal identity map from the metric space $(Z,d)$ to $Z$
is [uniformly] continuous.
A metric which is continuous and
dominating is said to be {\it admissible}.

The proofs of the two theorems exploit Hartman-Mycielski space
$\HM(X)$ of all $X$-valued step functions defined
on the interval $[0,1)$ (in a similar way as Zarichnyi \cite{Za}
applied the space of all $X$-valued measurable functions)
and also Pikhurko's \cite{Pi} idea
of constructing the required operator $T$
as sum of a series of operators ``separating'' points of $Y$.

Theorems 1 and 2 will be applied to construct linear operators extending
invariant metrics. For a topological space $X$ by $C(X\times X)$ we denote
the linear lattice of continuous functions on $X\times X$, equipped with the
compact-open topology. If a compact topological group $G$ acts on $X$, let
$$C_{inv}(X\times X)=\{f\in C(X\times X): f(gx)=f(gy)\mbox{ for all $g\in G$ and
$x,y\in X$}\}$$ denote the subspace of $C(X\times X)$ consisting of all
$G$-invariant functions.

\begin{theorem}\label{T3} Suppose a compact topological group $G$ acts on a
stratifiable space $Y$, and $X$ is a $G$-invariant subspace of $Y$
with $|X| \ge 2$. There exists a positive linear
continuous (in the compact-open topology) extension operator
$T:C_{inv}(X\times X)\to C_{inv}(Y\times Y)$
preserving constant functions, bounded invariant functions, invariant
pseudometrics, invariant metrics.
If the space $Y$ is metrizable, then additionally $T$ preserves
admissible metrics.
If the group $G$ is finite, then the operator $T$ is continuous with
respect to the point-wise convergence of functions.
\end{theorem}

The last theorem is an improvement obtained by the second author
of a former result of \cite{Ba0}
thanks to a discussion with C.~Atkin.
Another contribution of the second author is Section 5 containing
a relatively simple
construction of extension operators $S$, $S_1$, $S_2$ having almost all
properties of the operator $T$ from Theorems 1 and 2
(except that $S$ does not
preserve metrics, $S_1$ fails to preserve constants, and $S_2$ is not
positive).

\section{Hartman-Mycielski Construction}

This construction appeared in \cite{HM} in connection with some problems of
topological algebra, see also \cite{BM}.
For an $n\in\IN$ and a topological space $X$ let $\HM_n(X)$ be the set of
all functions $f:[0,1)\to X$ for which there exists a sequence
$0=a_0<a_1<\dots<a_n=1$ such that $f$ is constant on each interval
$[a_{i-1},a_i)$, $1\le i\le n$. Let $\HM(X)=\bigcup_{n\in\IN}\HM_n(X)$.

A neighborhood sub-base of the topology of $\HM(X)$ at an $f\in \HM(X)$
consists of sets $N(a,b,V,\e)$, where
\begin{itemize}
\item[1)] $0\le a<b\le 1$, $f$ is constant on $[a,b)$, $V$ is a
neighborhood of $f(a)$ in $X$, and $\e>0$;
\item[2)] $g\in N(a,b,V,\e)$ means that $|\{t\in[a,b):g(t)\notin
V\}|<\e$, where $|\cdot|$ denotes the Lebesgue measure.
\end{itemize}

As noted in \cite[Proposition 2]{BM} for every subspace $A$ of $X$,
the space $\HM(A)$ can be considered as a subspace of $\HM(X)$.
Also, the space $X$ can be identified with the subspace $\HM_1(X)$ of $\HM(X)$.

For an element $f\in \HM(X)$ let $\supp(f)$ denote the smallest subset
$A\subset X$ such that $f\in \HM(A)\subset \HM(X)$. Evidently that
$\supp(f)=f([0,1))$.

Recall that for a space $Z$ the space $\IR^Z$ is endowed with the
Tychonoff product topology
(which corresponds to the point-wise convergence on $\IR^Z$ considered as
a function space).

\begin{proposition}\label{P1} The formula
$$
hm(d)(f,g)=\int_0^1d(f(t),g(t))dt
$$
defines a positive linear continuous extension operator
$$hm:\IR^{X\times X}\to\IR^{\HM(X)\times \HM(X)}$$
preserving
constant functions, bounded functions, and bounded continuous
functions, pseudometrics, metrics,
dominating metrics, and bounded admissible metrics.
Moreover, for any totally bounded pseudometric $d$ on $X$ the
pseudometric $hm(d)$ is totally bounded on each $\HM_n(X)$, $n\in\IN$.
\end{proposition}

\begin{proof} It is an easy exercise to show that $hm$ is a positive
linear continuous extension operator preserving constant functions,
bounded functions,
and [pseudo]metrics. From Proposition 5 of \cite{BM} and its proof it
follows that $hm$ preserves dominating metrics and bounded admissible
metrics.

Let us show that $hm$ preserves bounded continuous functions. For
this fix a bounded continuous function $d:X\times X\to \IR$, $\e>0$ and
two elements $f,g\in \HM(X)$. Without loss of generality,
$|d(x,x')|\le 1$ for every $x,x'\in X$. Let $0=a_0<a_1<\dots<a_n=1$ be
a sequence such that both $f$ and $g$ are constant on each interval
$[a_{i-1},a_i)$, $1\le i\le n$. Using the continuity of $d$, for
every $i\in\{1,\dots,n\}$ pick neighborhoods $U_i,V_i\subset X$ of
$f(a_i)$, $g(a_i)$, respectively, such that for every $x\in U_i$,
$y\in V_i$ we have $|d(x,y)-d(f(a_i),g(a_i))|<\e/2$. Then
$U=\bigcap_{i=1}^nN(a_{i-1},a_i,U_i,\frac\e{8n})$ and
$V=\bigcap_{i=1}^nN(a_{i-1},a_i,V_i,\frac\e{8n})$ are neighborhoods
of $f$ and $g$, respectively, such that for every $f'\in U$, $g'\in
V$ we have $|hm(d)(f',g')-hm(d)(f,g)|<\e$. That means  the
function $hm(d):\HM(X)\times \HM(X)\to \IR$ is continuous.

Finally, we show that for every totally bounded pseudometric $d$
on $X$ the pseudometric $hm(d)$ is totally bounded on each $\HM_n(X)$.
Fix $n\in\IN$ and a totally bounded pseudometric  $d$ on $X$.
Consider the equivalence relation $\sim$ on $X$, where $x\sim y$ if
$d(x,y)=0$. Then the pseudometric $d$ induces a totally bounded
metric $\rho$ on the quotient space $X/$\hskip-2pt$\sim$. Let $(\tilde
X,\tilde\rho)$ denote the completion of $X/$\hskip-2pt$\sim$ by
the metric $\rho$
and let $p:X\to X/$\hskip-2pt$\sim\;\subset\tilde X$ be the quotient map. Clearly,
the space $\tilde X$ is compact. Then the space $\HM_n(\tilde X)$ is
compact as a continuous image of the product $\triangle^{n-1}\times
\tilde X^n$, where $\triangle^{n-1}=\{(a_0,\dots,a_n):0=a_0\le
a_1\le\dots\le a_n=1\}$ is an $(n-1)$-dimensional simplex, see
\cite[p.217]{BM}. The metric $hm(\tilde \rho)$, being
continuous, is totally bounded on $\HM_n(\tilde X)$. Since for every
$f,g\in \HM_n(X)$ \ $p\circ f,p\circ g\in \HM_n(\tilde X)$ and
$hm(d)(f,g)=hm(\tilde \rho)(p\circ f,p\circ g)$, we get $hm(d)$ is a
totally bounded pseudometric on $\HM_n(X)$.\qed
\end{proof}

For a space $X$ by $\exp_\omega X$ we denote the set of all finite
subsets of $X$. A map $u:Y\to \exp_\omega X$ is called {\it
upper-semicontinuous}
provided for every open set $U\subset X$ the set $\{y\in Y\mid u(y)\subset
U\}$ is open in $Y$.

Next, we prove that the spaces $\HM(X)$ over stratifiable spaces
have an important extension property. Below for a metric $d$ on a space
$X$ the open $d$-ball $\{x'\in X:d(x',x)<\e\}$ of radius $\e$
around a point $x\in X$ is denoted by $O_d(x,\e)$.

\begin{proposition}\label{P2} For every closed subset $X$ of a stratifiable space
$Y$ there exist
\begin{enumerate}
\item[1)] an upper semi-continuous map $u:Y\to \exp_\omega X$
such that $u(x)=\{x\}$ and
\item[2)] a continuous map $h:Y\to \HM(X)$ extending the
identity embedding $X\hookrightarrow \HM(X)$
such that $\supp(h(y))\subset u(y)$ for every $y\in Y$.
\end{enumerate}
Moreover, if $\dim(Y\setminus X)< n$, then $h(Y)\subset \HM_n(X)$.
If $d$ is an admissible metric for $Y$, then the map $u$ can be chosen so
that $u(y)\subset O_d(y,2d(y,X))$ for every $y\in Y$.
\end{proposition}

\begin{proof} Suppose $X$ is a closed subset of a stratifiable space $Y$.
By the proof of Theorem 4.3 of \cite{Bo}, there exists
a locally finite open cover $\U$ of $Y\setminus X$ and a
map $\alpha:\U\to X$ such that the map $u:Y\to \exp_\omega(X)$ defined by
$u(y)=\{y\}$ for $y\in X$
and
$u(y)=\{\alpha(U)\mid y\in \operatorname{cl}(U),\; U\in\U\}$ for $y\in
Y\setminus X$
is upper semi-continuous.
Let $\le$ be any linear ordering of the set $\U$ and let
$\{\lambda_U:Y\setminus X\to
[0,1]\}_{U\in\U}$ be a partition of unity subordinate to the cover $\U$.
For a $y\in Y\setminus X$ define a function $h(y)\in \HM(X)$ letting
$$
h(y)(t)=\alpha(U), \text{ \ if \ } \sum_{V<U}\lambda_V(y)\le t<
\sum_{V\le U}\lambda_V(y).
$$
Because only finitely many of $\lambda_V(y)$'s are distinct from
zero, the function $h(y)$ is well-defined.

For $y\in X$ let $h(y)=y\in X\subset \HM(X)$.

We claim that the so-defined map $h:Y\to \HM(X)$ is continuous and
satisfies the requirements of Proposition 2. The inclusion
$supp(h(y))\subset u(y)$, $y\in Y$, follows from the definitions of
$h(y)$ and $u(y)$.

The continuity of $h$ on the set $Y\setminus X$ easily follows from the
local finiteness of the cover $\U$.
Let us
verify the continuity of $h$ at a point $x\in X$. Fix any
neighborhood $U$ of $h(x)=x$ in $\HM(X)$. According to the
definition of the topology of $\HM(X)$, there exists a
neighborhood $V$ of $x$ in $X$ such that $\HM(V)\subset U$.
Since the map $u:Y\to \exp_\omega X$ is upper-semicontinuous and $u(x)=\{x\}$,
there is a neighborhood $W$ of $x$ in $Y$ such that $u(y)\subset V$
for every $y\in W$.
Then for such $y$ we have $h(y)\in \HM(u(y))\subset \HM(V)\subset U$,
i.e. $h$ is continuous at the point $x$.

If $\dim(Y\setminus X)< n$ then the
cover $\U$ can be chosen to be of order $\le n$.
In this case, according to the construction, we get
 $h(Y)\subset \HM_n(X)$.

If $Y$ is a metrizable space with an admissible metric $d$, then using
the classical technique of Dugundji \cite{Du} we may construct the map $u$
so that $u(y)\subset O_d(y,2d(y,X))$ for every $y\in Y$.
\end{proof}

\begin{question} Is $\HM(X)$ an absolute extensor for stratifiable spaces?
The answer is ``yes'' for separable metrizable $X$. (This
can be shown applying the arguments of \cite[Ch.VI, \S7]{BP}).
\end{question}

\section{Construction of an extension operator $T$}

Suppose $X$ is a closed subset of a stratifiable space $Y$ and $a,b$ be
two distinct points of $X$.
An operator $T$ satisfying the requirements of Theorems 1 and 2
will be constructed as the sum of a series
$\sum_{n=1}^\infty 2^{-n}T_n$, where the collection of extension operators
$\{T_n:\IR^{X\times X}\to \IR^{Y\times Y}\}_{n=1}^\infty$  ``separates"
points of  $Y$.

It is  known that every stratifiable space admits a bijective
continuous map onto a metrizable space (combine [Bo, Lemma
8.2] with [Bo, property (A) on p.2]).
Therefore, there is a continuous metric $d\le 1$ on $Y$.
Moreover, applying Theorem 5.2 of
\cite{Bo}, we may adjust the metric $d$ so that $d(y,X)>0$ for every $y\in
Y\setminus X$, where, as usual, $d(y,X)=\inf\{d(y,x): x\in X\}$.
If $Y$ is a metrizable uniform space, then $d$ will be assumed
to generate the uniformity of $Y$.

Let $h:Y\to \HM(X)$ and $u:Y\to\exp_\omega(X)$ be the maps from Proposition 2
(in case $\dim Y\setminus X<\infty$ we assume that $h(Y)\subset \HM_k(X)$
for some $k\in\IN$).

For every $n\in\IN$ we shall define
an extension operator $T_n:\IR^{X\times X}\to\IR^{Y\times Y}$ as follows.
Fix $n\in\IN$.
Let $\U_n$ be a locally finite (resp. finite,
if the metric $d$ is totally bounded) open
cover  of the space $Y$ such that $\diam_d(U)< 2^{-n}$ for
every  $U\in \U_n$, and let $\{\lambda^n_U:Y\to [0,1]\}_{U\in\U_n}$ be a
partition of unity, subordinate to the cover $\U_n$.
Further we consider $\U_n$ as a discrete topological space. Let $\le $ be any
linear ordering on $\U_n$ and let $h_n:Y\to \HM(\U_n)$ be the map defined for
a $y\in Y$ by the formula
$$
h_n(y)(t)=U,\text{ \ if \ } \sum_{V<U}\lambda_V^n(y)\le t<\sum_{V\le U}\lambda^n_V(y).
$$
As in the proof of Proposition~\ref{P2}, it can be shown that the map $h_n$ is
continuous.

By $X\sqcup \U$  denote the disjoint union of the
spaces $X$ and $\U_n$, $n\in\IN$.
According to \cite[Proposition~2]{BM}, we may identify $\HM(X)$ and $\HM(\U_n)$
with subspaces of $\HM(X\sqcup \U)$.
Finally, define a map $f_n:Y\to \HM(X\sqcup\U)$
letting for a $y\in Y$
$$
f_n(y)(t)=\begin{cases} h_n(y)(t),&\mbox{if \;\; $0\le t<\min\{1,n\,d(y,X)\}$};\\
h(y)(t),&\mbox{if \;\; $\min\{1,n\,d(y,X)\}\le t<1$}.
\end{cases}
$$
It is easily seen that $f_n$ is a continuous map
extending the natural embedding
$X\subset \HM(X)\subset \HM(X\sqcup\U)$.

Let us consider the linear
operator $E:\IR^{X\times X}\to \IR^{(X\sqcup\U)\times (X\sqcup \U)}$ defined for
every $p\in \IR^{X\times X}$ by
$$
E(p)(x,y)=\begin{cases}
p(x,y),&\mbox{if $x,y\in X$;}\\
\frac12p(x,a)+\frac12p(x,b),&\mbox{if $x\in X,\; y\in\U$;}\\
\frac12p(a,y)+\frac12p(b,y),&\mbox{if $x\in \U,\; y\in X$;}\\
p(a,b),&\mbox{if $x,y\in \U$ and $x\ne y$}\\
0, & \mbox{if $x=y$};
\end{cases}
$$
(recall that $a,b$ are two fixed point in $X$).
One can easily verify that  $E$ is a positive linear
continuous extension operator
preserving constants, bounded, bounded continuous functions and
[pseudo]metrics.

The operator $T_n:\IR^{X\times X}\to\IR^{Y\times Y}$ is defined as the composition
$$
\IR^{X\times X}\overset{E}\to\longrightarrow\IR^{(X\sqcup\U)\times(X\sqcup\U)}
\overset{hm}\to\longrightarrow \IR^{\HM(X\sqcup\U)\times \HM(X\sqcup\U)}
\overset{(f_n\times f_n)^*}\to\longrightarrow\IR^{Y\times Y},
$$
where $(f_n\times f_n)^*(p)=p\circ (f_n\times f_n)$ for $p\in
\IR^{\HM(X\sqcup\U)\times \HM(X\sqcup\U)}$, equivalently, by the
explicit formula
$$
T_n(p)(y,y')= \int_0^1E(p)(f_n(y)(t),f_n(y')(t))\,dt \quad
\hbox{{\sl for}} \quad p\in \IR^{X\times X}, \quad y,y' \in Y.
$$

Remark that $T_n$ is a positive linear continuous extension operator
preserving constants, bounded, bounded continuous functions
and pseudometrics.

Finally, we define the required operator $T:\IR^{X\times X}\to \IR^{Y\times Y}$
by the formula
$$T=\sum_{n=1}^\infty \frac1{2^n}T_n.$$

We shall verify the properties of the operator $T$.
First, observe that the definition of $T$ is correct, i.e. for every
function $p:X\times X\to \IR$ and every $y,y'\in Y$ the series
$\sum_{n=1}^\infty 2^{-n}T_n(p)(y,y')$ is convergent. This is
trivial, when $y,y'\in X$ (all $T_n$'s are extension operators).
If $y\in X$ and $y'\notin X$ then for every
$n\in\IN$ with $d(y',X)\ge\frac1n$, by the construction of $T_n$, we have
$T_n(p)(y,y')=\frac12p(y,a)+\frac12 p(y,b)$. If $y,y'\notin X$ then,
for every $n\in\IN$ with $d(y,X),d(y',X)\ge\frac1n$, we have
$|T_n(p)(y,y')|\le |p(a,b)|$. These remarks imply that the series
$\sum_{n=1}^\infty 2^{-n}T_n(y,y')$  converges for every $y,y'\in Y$,
i.e. the definition of $T$ is correct.

Since $T_n$'s are positive linear extension operators preserving constants,
bounded functions, bounded continuous functions
functions and pseudometrics, so is the operator $T$.

\section{Proof of Theorem~\ref{T1}}

In an obvious way Theorem 1 follows from the above-mentioned properties of
the operator $T$ and the subsequent four lemmas.
The first of them can be easily derived from the construction of $T$.

\begin{lemma}\label{L1} Let $y,y'\in Y$ and $A=\supp(h(y))\cup\supp(h(y'))\cup\{a,b\}$.
If $p,p':X\times X\to\IR$ satisfy $p|A\times A\le p'|A\times A$, then
$T(p)(y,y')\le T(p')(y,y')$. Moreover, if $p|A\times A\equiv c$, then
$T(p)(y,y')=c$.
\end{lemma}

\begin{lemma}\label{L2} The operator $T:\IR^{X\times X}\to\IR^{Y\times Y}$ is continuous
with respect to the uniform, point-wise or compact-open topologies on the
function spaces $\IR^{X\times X}$ and $\IR^{Y\times Y}$.
\end{lemma}

\begin{proof} Because the operator $T$ is positive and preserves constant
functions, it is continuous
with respect to the uniform convergence of functions.

Let us show that the operator $T$ is continuous with respect to the
point-wise convergence of functions. For this, fix points $y,y'\in Y$ and
notice that  the set $A=\{a,b\}\cup\supp(h(y))\cup\supp(h(y'))$ is finite.
By Lemma 1,  for a function
$p:X\times X\to \IR$ the inequality $|p(x,x')|\le 1$ for
every $(x,x')\in A\times A$ implies $|T(p)(y,y')|\le 1$. This means that the
operator $T$ is continuous with respect to the point-wise convergence of
functions.

To show that  $T$ is continuous with respect to
the compact-open topology fix a compactum $C\subset Y\times Y$ and notice that
the set $K'=\bigcup\{u(y)\mid y\in \pr_1(C)\cup \pr_2(C)\}\subset X$ is
compact because of the upper-semicontinuity of the map $u:Y\to\exp_\omega X$
(see \cite[Theorem VI.7.10]{FF}) (by $\pr_i:Y\times Y\to Y$
we denote the projection onto the
corresponding factor). Consider the compact set $K=K'\cup\{a,b\}$. Then
 $\supp(h(y))\cup
\supp(h(y'))\subset u(y)\cup u(y')\subset K$ for every $(y,y')\in C$. Now
Lemma~\ref{L1} yields that for a function $p:X\times X\to \IR$ if $|p(x,x')|\le 1$
for every $(x,x')\in K\times K$ then $|T(p)(y,y')|\le 1$ for every $(y,y')\in
C$. But this means that the operator $T$ is continuous in the compact-open
topology. 
\end{proof}

\begin{lemma}\label{L3} The operator $T$ preserves continuous functions.
\end{lemma}

\begin{proof} Let $p:X\times X\to \IR$ be a continuous function. Fix any point
$(y_0,y_0')\in Y\times Y$. Let $M=\max\{|p(x,x')|: x,x'\in \{a,b\}\cup u(y_0)\cup
u(y_0')\}$. Since the map $p$ is continuous, there is a neighborhood $U\subset X$
of the compactum $\{a,b\}\cup u(y_0)\cup u(y_0')$ such that $|p(x,x')|<M+1$
for every $x,x'\in U$. Since the map $u:Y\to\exp_\omega X$ is upper-semicontinuous,
there are neighborhoods $V,V'\subset Y$ of $y_0,y_0'$ respectively such that
for every $y\in V$ and $y'\in V'$ we have $u(y)\cup u(y')\subset U$.

Now consider the bounded continuous function $\tilde p:X\times X\to \IR$
defined by the formula
$$
\tilde p(x,x')=\begin{cases} p(x,x'), & \mbox{ if } -M-1\le p(x,x')\le M+1\\
M+1, &\mbox{ if } p(x,x')\ge M+1\\
-M-1, &\mbox{ if } p(x,x')\le -M-1.\end{cases}
$$
Obviously that $\tilde p|U= p|U$. Moreover, since the operator
$T$ preserves bounded continuous functions, the map $T(\tilde
p):Y\times Y\to\IR$ is continuous. Now remark that for every $(y,y')\in V\times V'$ \
$\supp(h(y))\cup\supp(h(y'))\subset \{a,b\}\cup u(y)\cup u(y')\subset U$. Since $\tilde
p|U\times U=p|U\times U$, by Lemma 1, $T(p)(y,y')=T(\tilde p)(y,y')$. Therefore,
$T(p)|V\times V'=T(\tilde p)|V\times V'$ and the function $T(p)$ is continuous.
\end{proof}

\begin{lemma}\label{L4} The operator $T$ preserves metrics.
\end{lemma}

\begin{proof} Let $p$ be a metric on $X$. Since the operator $T$ preserves
pseudometrics, it remains to prove  that
$T(p)(y,y')\not= 0$ for distinct $y,y'\in Y$. So, fix $y,y'\in Y$ with
$y\not=y'$.

If $y,y'\in X$ then $T(p)(y,y')=p(y,y')\ne 0$ because $p$ is a metric on
$X$. Now assume that $y\in X$ and $y'\notin X$. Then $d(y',X)>\frac1n$
for some $n\in\IN$. Consequently,
$f_n(y)=y\in X\subset \HM(X\sqcup\U)$ and $f_n(y')=h_n(y')\in
\HM(\U_n)\subset \HM(X\sqcup \U)$. By the property of the operator $E$, we
have $E(p)(y,h_n(y')(t))=\frac12(p(y,a)+p(y,b))\ge\frac12p(a,b)$ for every
$t\in[0,1)$ and thus
$$
T_n(p)(y,y')=
\int_0^1E(p)(y,h_n(y')(t))dt\ge\frac12 p(a,b)>0.$$
This yields
$T(p)(y,y') \ge 2^{-n}T_n(p)(y,y') > 0$.

Now assume that $y,y'\in Y\setminus X$. Then there is an $n\in\IN$ such that
$d(y,X)>n^{-1}$, $d(y',X)>n^{-1}$ and $d(y,y')>2^{-n+1}$.
In this case, $f_n(y)=h_n(y)$ and $f_n(y')=h_n(y')$.
Since $\diam(U)<2^{-n}$
for $U\in\U_n$, there is no $U\in \U_n$ with $\{y,y'\}\subset U$.
Consequently, $\supp(h_n(y))\cap \supp(h_n(y'))=\emptyset$. By the
definition of the metric $E(p)$, $E(p)(h_n(y)(t),h_n(y')(t))=p(a,b)$ for
every $t\in[0,1)$.  Then
$$
2^nT(p)(y,y')\ge T_n(p)(y,y')=
\int_0^1E(p)(h_n(y)(t),h_n(y')(t))dt=p(a,b)>0
$$
Therefore, $T(p)$ is
a metric on $Y$.
\end{proof}

\section{Proof of Theorem~\ref{T2}}

In this section we suppose that $Y$ is a metrizable uniform
space and the metric $d$ generates the uniformity of $Y$. Moreover, the
map
$u$ constructed in Proposition~\ref{P2} has the following property: $u(y)\subset
O_d(y,2d(y,X))$ for every $y\in Y$.

In an obvious way Theorem~\ref{T2} follows from Theorem~\ref{T1} and the subsequent four
lemmas.

\begin{lemma}\label{L5} The operator $T$ preserves the class of
dominating metrics.
\end{lemma}

\begin{proof}
Let $p$ be a dominating metric for $X$.
To show that the metric $T(p)$ dominates the topology of $Y$, it suffices
for every $y\in Y$ and every $\e\in(0,1]$ to find $\delta>0$ such that
$T(p)(y,y')\ge\delta$ for every $y'\in Y$ with $d(y',y)>\e$.

First we consider the case $y\notin X$. Then we can find $n\in\IN$ such
that $d(y,X)>\frac1n$ and $2^{-n+1}<\e$. Let $\delta=2^{-n-1}p(a,b)$ and
$y'\in Y$ be any point with $d(y,y')>\e$. Then $d(y,y')>2^{-n+1}$ and by
the choice of the cover $\U_n$ and the map $h_n$, we have
$\supp(h_n(y))\cap\supp(h_n(y'))=\emptyset$. As we have
observed in the proof
of Lemma 4, this implies $E(p)(h_n(y')(t),h_n(y')(t))=p(a,b)$ for every
$t\in[0,1)$. Besides, it follows that
$E(p)(h_n(y)(t),h(y')(t))\ge\frac12p(a,b)$. Then
$$
2^nT(p)(y,y')\ge T_n(p)(y,y')=
\int_0^1 E(p)(h_n(y)(t),f_n(y')(t))dt\ge
\frac12p(a,b)=2^n\delta.
$$

Now assume that $y\in X$. Let $n\in\IN$ be such that $2^{-n+1}<\e$. Since
the metric $p$ is dominating for $X$, there is $\eta>0$ such that
$p(y,x)>\eta$ for every $x\in X$ with $d(y,x)>\e$.
Let $\delta=\min\{2^{-n-1}p(a,b), n\e2^{-n-3} p(a,b), 3\eta/8\}$ and fix any
point $y'\in Y$ with $d(y,y')>\e$. To verify that $T(p)(y,y')\ge\delta$,
consider two cases:

1) $d(y',X)\ge\frac\e4$. By the property of the metric
$E(p)$, we have $E(p)(y,h_n(y')(t))\ge\frac12p(a,b)$ for every
$t\in[0,1)$. Then
$$
\begin{aligned}
2^nT(p)(y,y')& \ge T_n(p)(y,y')=\int_0^1E(p)(y,f_n(y')(t))dt
\ge\int_0^{\min\{1,d(y',X)\}}E(p)(y,h_n(y')(t))dt\ge\\
&\ge
\min\{1,n\,d(y',X)\}\frac12p(a,b)>
2^{-1}\min\{1,{n\e}/4\}p(a,b) \ge 2^n\delta.
\end{aligned}
$$

Now pass to the other case:

2) $d(y',X)<\frac\e4$. Then
$\supp(h(y'))\subset u(y')\subset O_d(y',\frac\e2)$
and because $d(y,y')>\e$, we get
$d(y,h(y')(t))>\frac\e2$ for every $t\in[0,1)$.
By the choice of $\eta$,
this implies $p(y,h(y')(t))>\eta$ for every $t\in[0,1)$. Then
$$
\begin{aligned}
2T(p)(y,y')&\ge T_1(p)(y,y')
=\int_0^1 E(p)(y,f_1(y')(t))dt \ge \\
&\ge\int_{d(y',X)}^1 E(p)(y,h(y')(t))dt
\ge\int^1_{\e/4}p(y,h(y')(t))dt
\ge(1-\frac\e4)\eta\ge\frac34\cdot\eta\ge 2\delta.
\end{aligned}
$$
\end{proof}

\begin{lemma}\label{L6} The operator $T$ preserves uniformly dominating
metrics.
\end{lemma}

\begin{proof} Let $p$ be a uniformly dominating metric for the uniform
space $X$. To show that the metric $T(p)$ is uniformly dominating for $Y$,
it suffices to verify that the formal identity map $(Y,T(p))\to
(Y,d)$ between the respective metric spaces is uniformly continuous.

Fix any $\e>0$. We have to find $\delta>0$ such that for every $y_1,y_2\in
Y$ the inequality $T(p)(y_1,y_2)<\delta$ implies $d(y_1,y_2)<\e$;
equivalently, $d(y_1,y_2)\ge\e$ implies $T(p)(y_1,y_2)\ge\delta$. To find
such $\delta$, select $n\in\IN$ so that $2^{-n}<\frac\e2$ and
$\frac{n}{2^{n+2}}<\frac12$. Since the metric $p$ is uniformly dominating
for $X$, there exists $\delta>0$ such that $d(x,x')<\frac\e2$ for all
$x,x'\in X$ with $p(x,x')<2^{n+1}\delta$. Moreover, we may take $\delta $
so small that $2^n\delta<\frac{n}{2^{n+3}}p(a,b)$.

We claim that the so-chosen number $\delta$ satisfies our requirements.
To show this, fix any points $y_1,y_2\in Y$ with $d(y_1,y_2)\ge\e$.
Because $T(p)(y_1,y_2)\ge 2^{-n}T_n(y_1,y_2)$, it suffices to verify the
inequality $T_n(y_1,y_2)\ge 2^{n}\delta$. Two cases will be considered
separately:

1) $\max\{d(y_1,X),d(y_2,X)\}\ge2^{-n-2}$. Without loss of
generality, $d(y_1,X)\le d(y_2,X)$. Since $d(y_1,y_2)\ge\e$ and
$\sup_{U\in\U_n}\diam(U) <2^{-n}<\frac\e2$, we get
$\supp(h_n(y_1))\cap\supp(h_n(y_2))=\emptyset$. It follows that

\noindent
$E(p)(h_n(y_1)(t),h_n(y_2)(t))\ge p(a,b)$  and
$E(p)(h_n(y_2),h(y_1))\ge\frac12 p(a,b)$  for every
$t\in[0,1)$.

\noindent
Then
$$
\begin{aligned}
T_n(p)(y_1,y_2)&= \int_0^1E(p)(f_n(y_1)(t),f_n(y')(y_2)(t))dt\ge\\
&\ge\frac12\min\{1,nd(y_2,X)\}p(a,b)\ge\frac12p(a,b)\frac n{2^{n+2}}>2^n\delta.
\end{aligned}
$$

Next, we consider the case:

2) $\max\{d(y_1,X),d(y_2,X)\}<{2^{-n-2}}$.
By the definition of the map $u$, we have $\supp(h(y_i))\subset u(y_i)\subset
O_d(y_i,2d(y_i,X))\subset O_d(y_i,\frac\e4)$ for $i=1,2$. Consequently,
$$
d(h(y_1)(t),h(y_2)(t))>\frac\e2 \text{ \ for every \ }
t\in[0,1)
$$
and by the choice
of $\delta$, we
get $p(h(y_1)(t),h(y_2)(t))>2^{n+1}\delta$. Finally, for the pseudometric
$T_n(p)$ we obtain
$$
\begin{aligned}
T_n(p)(y_1,y_2)&=\int_0^1E(p)(f_n(y_1)(t),f_n(y_2)(t))dt\ge
\int^1_{1-n\max\{d(y_1,X),d(y_2,X)\}}p(h(y_1)(t),h(y_2)(t))dt\ge\\
&\ge(1-n\max\{d(y_1,X),d(y_2,X)\})
2^{n+1}\delta\ge (1-n2^{-n-2})2^{n+1}\delta\ge\frac12 2^{n+1}\delta=2^n\delta.
\end{aligned}
$$
\end{proof}


\begin{lemma}\label{L7} If the uniform space $Y$ is complete, then the operator
$T$ preserves complete continuous uniformly dominating metrics.
\end{lemma}

\begin{proof} Observe that in a complete metrizable uniform space
every continuous uniformly dominating metric is complete and
apply Lemmas 3 and 6.
\end{proof}

\begin{lemma}\label{L8} If the uniform space $Y$ is totally bounded and
$\dim Y\setminus X<\infty$ then the operator $T$ preserves
totally bounded pseudometrics.
\end{lemma}

\begin{proof} Fix a totally bounded pseudometric $p$ on $X$.
It is enough to show that each pseudometric $T_n(p)$ is totally bounded.
Fix any $n\in\IN$. Since the metric $d$ on $Y$ is totally
bounded, by the construction, the cover $\U_n$ is finite.
Then the metric $E(p)$
on $X\sqcup \U_n$ is totally bounded.
Let $k>|\U_n|$ be such that $h(Y)\subset
\HM_k(X)$. Then $f_n(Y)\subset \HM_{2k}(X\sqcup \U_n)$ and $T_n(p)(y,y')=
hm(E(p))(f_n(y),f_n(y'))$ for every $y,y'\in Y$. By
Proposition 1, the pseudometric $hm(E(p))$ is totally bounded
on $\HM_{2k}(X\sqcup\U_n)$. Hence,
the  pseudometric $T_n(p)$ is totally bounded on $Y$.
\end{proof}

\section{Proof of Theorem~\ref{T3}}

Assume that $G$ is a compact group, $\mu$ the Haar measure on
$G$, $Y$ is a (left) $G$-space and $X$ is a closed subspace of
$Y$ consisting of at least two points and invariant under the
action of $G$.
For $Z \in \{X,Y\}$ by $C(Z\times Z)$ we denote
the linear lattice of continuous functions on $Z\times Z$,
equipped with the compact-open topology and by
$C_{inv}(Z\times Z)$ its linear subspace
consisting of continuous {\it invariant} functions, i.e., such
that $f(gx,gy)=f(x,y)$ for every $g\in G$ and $x,y\in X$.

\begin{proposition}\label{P3} The averaging operator $A:C(Y\times Y)\to
C_{inv}(Y\times Y)$ defined by
$$
Af(y,y') = \int_G f(gy,gy')d\mu \text{ \ for\ }
f \in C(Y \times Y),\quad y,y' \in Y
$$
is a continuous
retraction of $C(Y \times Y)$ onto $C_{inv}(Y \times Y)$.
The operator $A$ takes constants, pseudometrics, metrics,
admissible metrics into
constants, invariant pseudometrics, invariant metrics,
invariant admissible metrics, respectively.
\end{proposition}

\begin{proof} Let $d \in C(Y \times Y)$ be a metric on $Y$ and
$d' = Ad$. Let $a,b \in Y$, $a\neq b$. There is a neighborhood
$U$ of the neutral element of the group $G$ such that
$d(ga,gb) \ge 2^{-1} d(a,b)$ for $g \in U$. Therefore
$d'(a,b) \ge 2^{-1} \mu(U) d(a,b) > 0$, i.e., $d' = Ad$ is a
metric.

Now assume that $d \in C(Y \times Y)$ is an admissible metric, and
$(y_n) $ is a sequence of points of $Y$ such that
$\limn d'(y_n,y) = 0$ for some $y\in Y$.
Hence the sequence of real functions
$\varphi_n(g)= d(gy_n,gy)$ tends to zero in the $L_1$-norm, and
since $\mu(G) =1 < \infty$, there is a subsequence
$\varphi_{k_n}$ which tends to zero almost everywhere, in
particular, $\limn d(g_0y_{k_n},g_0y) = 0$ for some $g_0 \in G$.
``Multiplying the last relation from the left'' by $g_0^{-1}$ we
get $\limn d(y_{k_n},y) = 0$. The same arguments yield that
every subsequence of the sequence $(y_n)$ contains a subsequence
convergent (in the admissible metric $d$) to $y$. That means
that the whole sequence $(y_n)$ tends to $y$. We have proved
that the $d'= Ad$ is dominating, and (being continuous) is
admissible.
 The other assertions of the proposition are evident.
\end{proof}

\begin{proof}[Proof of Theorem~\ref{T3}] Let $T$ be the operator appearing in Theorem 1
(Theorem 2 in case of metrizable $Y$).
The required operator $I$ is the composition
$I = A \circ T|C_{inv}(X \times X)$.
\end{proof}

\section{The extension operators
$S,\,S_1,\, S_2 :\IR^{X\times X}\to \IR^{Y\times Y}$}

In this section we present a simple construction of extension operators
$S$, $S_1$, $S_2$ having almost all properties of the operator $T$.

Suppose $Y$ is a stratifiable space, $X$ is a closed subset of $Y$ and
$a,b$ are two distinct points in $X$. As we said, the space $Y$ admits a
continuous metric $d\le 1$ such that $d(y,X)>0$ for all $y\in Y\setminus X$. If
$Y$ is metrizable, we assume that $d$ is an admissible metric for $Y$.

For $y,y'\in Y$ let
$$
d^*(y,y')=\min[d(y,y'), d(y,X)+d(y',X)],
$$
Clearly, $d^*$ is a continuous pseudometric on $Y$ (moreover, the
restriction of $d^*$ on $Y\setminus X$ is a metric). Let $h:Y\to \HM(X)$ be the
map appearing in Proposition 2 and define
$$\begin{aligned}
S(p)(y,y')= hm(p)(h(y),h(y')) = \int_0^1p(h(y)(t),h(y')(t))dt,\\
S_1(p) = S(p) + p(a,b)d^*, \quad S_2(p) = S(p) +
(p(a,b)-p(a,a))d^*
\end{aligned}
$$
for $p \in \IR^{X\times X}$, $y,y'\in Y$
Thus we have defined three extension operators $S,S_1,S_2:\IR^{X\times X}\to\IR^{Y\times Y}$.

\begin{theorem}\label{T4} The operators $S, S_1$ and $S_2$ satisfy the
requirements of Theorem 1 except that $S$ does not preserve metrics, $S_1$
fails to preserve constants and $S_2$ is not positive. Moreover, if the
space $Y$ is metrizable, then the operators $S_1$ and $S_2$ preserve
dominating and admissible metrics.
\end{theorem}

\begin{proof} The first statement of the theorem easily follows from
Propositions~\ref{P1} and \ref{P2} (to prove that these operators preserve continuous
functions one should apply the arguments from Lemma~\ref{L3}).
The fact that in the metric case, $S_1$ and $S_2$
preserve dominating metrics is an immediate consequence of the
next two easy lemmas.

\begin{lemma}\label{L9a} For every dominating metric $p$ on $X$ the pseudometric
$\rho=S(p)$ has the following property :
\item{($*$)} Let $y_n \in Y$ for $n \in \IN$ and  $x\in X$.
Then $\limn \rho(y_n,x) =0$ and $\limn d(y_n,X) = 0$ imply
$\limn d(y_n,x) = 0$.
\end{lemma}

\begin{proof} (cf. proof of Lemma~\ref{L5}). Recall that $d$ is a fixed
admissible metric for $Y$. According to the last assertion of
Proposition~\ref{P2} and the definition of the operator S,
for every $y \in Y$ there is an $y' \in u(y) \subset X$ such that
$$
d(y,y') \le 2d(y,x) \text{ \ \ and \ \ } p(y',x) \le \rho(y,x).
$$
We have $d(y_n,x) \le d(y_n,y'_n) +d(y'_n,x)
\le 2d(y_n,X) + d(y'_n,x)$.

But $0 \le p(y'_n,x) \le \rho(y_n,x)  \to  0$ as
$n\to\infty$,  and since $p$ is a dominating metric for $X$, we
get $\limn d(y'_n,x) = 0$, and by the assumption of ($*$),
$\limn d(y_n,x)=0$.
\end{proof}

\begin{lemma}\label{L9b} For every pseudometric $\rho$ in $Y$
and every constant $c>0$
the sum $\rho+cd^*$ is a dominating metric on $Y$, provided $\rho$ has the
property $(*)$.
\end{lemma}

\begin{proof} By ($*$), the sum $\rho+cd^*$ is dominating ``at
each point''$x \in X$. In order to show the domination at the
remaining points it is enough to examine the second term
$d^*$ which is a metric, when restricted to $Y\setminus X$. 
\end{proof}

\end{proof}

Finally, we pose an open problem suggested by Theorem~\ref{T2} and a known result
of J.S.~Isbell \cite{Is} according to which for every subspace $X$ of a uniform
space $Y$, every bounded uniformly continuous pseudometric on $X$ extends
to a bounded uniformly continuous pseudometric on $Y$.

\begin{problem} Suppose $X$ is a subspace of a metrizable uniform space
$Y$. Does there exist a ``nice'' operator extending bounded uniformly
continuous pseudometrics from $X$ over $Y$.
\end{problem}


\end{document}